\documentclass{amsart}
\usepackage{amsmath}
\usepackage{amsfonts}
\usepackage{amssymb}
\usepackage{graphicx}
\usepackage{epstopdf}
\usepackage{url}
\usepackage[section]{placeins}

\begin{document}
\title{A SYNTHETIC ALGEBRAIC APPROACH TO A DISCUSSION OF RIEMANN'S HYPOTHESIS}

\author{Michele Fanelli}
\email[Michele Fanelli]{michele31fanelli@gmail.com}
\address [Michele Fanelli]{Via L.B. Alberti 5, 20149 Milano, Italy}
\author{Alberto Fanelli}
\email[Alberto Fanelli]{apuntoeffe@gmail.com}
\address[Alberto Fanelli] {Via L.B. Alberti 12, 20149 Milano, Italy}
\date{Dic 16, 2013}
\newtheorem{defn}{Definition}
\newtheorem{cor}{Corollary}
\newtheorem{theorem}{Theorem}

%
%
\begin{abstract}
A fresh approach to the long debated question is proposed, starting from the GRAM-BACKLUND analytical continuation of the Zeta function (G-B Zeta expression). Consideration is given to the symmetric (even-exponent) and anti-symmetric (odd exponent) components of the power series representation of the G-B formulation along a circular path containing the 4 hypothetical zero-points predicted by HADAMARD and DE LA VALL\'EE POUSSIN (H/DLVP ‘outlying’ quartet). From the necessary conditions required of the even- and odd-exponent components at a representative zero-point of the hypothetical quartet some interesting logical consequences are derived and briefly discussed in the framework of the RIEMANN Hypothesis. A temporary mapping of the representative zero-point of the hypothetical quartet onto a higher-dimensional auxiliary domain provides an intriguing short-cut to the negative conclusion about the possibility of existence of such outlying zero-points.
\end{abstract}

\maketitle
\tableofcontents

%
%
\section{The Zeta function. Definition of trivial and non-trivial, c.c. zeroes. Riemann's hypothesis and the scope of the present essay.}\label{sec:1}

As well known, the formal definition of the Zeta function is: 
\begin{equation}
Z(s)=\underset{n=1}{\overset{\infty }{\sum }}\frac{1}{n^s} \label{1}
\end{equation}
where $s=X+i.Y \in \mathbb{C}$, , $X \in \mathbb{R}$ , $Y \in \mathbb{R}$ , $Z \in \mathbb{C}$ , $n \in \mathbb{I}$ (natural integer). Ref. \cite{Edwa74} is a good source.

The problem of the existence and location of the zeroes of the Zeta function is known to have important implications in Number Theory. These zeroes fall into one or the other of two main categories: "trivial" and "non-trivial" zeroes.
Trivial zeroes, infinite in number, are located on the real axis at $s=-2,-4,-6... -2.n...$ and do not concern the object of the present essay. Non-trivial known zeroes are also infinitely many and are located on the ‘critical line’ (C.L.) $X=\frac{1}{2}$. These will be termed in the following "canonical zeroes". They occur in complex conjugated (c.c.) pairs $s=\frac{1}{2}\pm i.Y$ with $Y\neq 0$. It is not yet known whether there are non-trivial zeroes outside the C.L., but it is known, see Ref.\cite{Edwa74}, that all non-trivial zeroes are constrained to lie inside to the "critical strip" $0<X<1$.
The celebrated RIEMANN Hypothesis, formulated in 1859, see Refs.\cite{Edwa74},\cite{Riem59}, states that those on the C.L. are the only non-trivial zeroes of the Zeta function. In other words, no pair of c.c. zeroes can be found outside of the C.L. The proof of the possibility, or the impossibility, of existence of these "outlying" pairs has not been reached up to now, although the known zeroes along the C.L. have been determined (and the search for outlying zeroes has been pursued), by numerical computing, up to extremely large values of Y, so that it is ascertained that the outlying zeroes, if they exist, are to be found for exceedingly high values of Y , see Refs. \cite{Gour04} and \cite{Odly01}. A proof (or the disproving) of the Hypothesis would have important implications not only in Number theory, but also in practical applications of the properties of primes, such as  cryptography, see Ref. \cite{Rive77}.
The aim of the present essay is to propose a synthetic, algebraic approach to the discussion of RIEMANN’s Hypothesis, using only the tools of ordinary Complex Algebra and Calculus.

%
%
\section{Main properties of the Zeta function in general and of the Gram-Backlund extension in particular}\label{sec:2}

The Zeta function is defined by Eq. \eqref{1} only over the half-plane $X>1$. In order to investigate the RIEMANN Hypothesis, it is necessary to extend the formula \eqref{1} by analytic continuation. One such extension is the GRAM-BACKLUND extension, see Refs. \cite{Gram03} and \cite{Back18}, which can be written as: 	

\begin{equation}
Z_{GB}(s)=\underset{n=1}{\overset{N>1}{\sum }}\frac{1}{n^s}+\frac{N^{1-s}}{s-1}+\frac{N^{-s}}{2}+s.\underset{\mu =1}{\overset{\infty }{\sum }}\frac{B_{2\mu }}{(2.\mu )!}.(s+1)\ldots (s+2.\mu -2).N^{-s-2.\mu +1} \label{2}
\end{equation}
where the symbols $B_{2\mu }$ stand for the BERNOULLI numbers of even index.
The Zeta function, and its GRAM-BACKLUND extension, are analytic functions of $s$. As such, they enjoy the properties of analytic functions, to wit: $Z\left(\overline{s}\right)=\overline{Z}(s)$, possibility to develop $Z(s)$ as a power-series of $s$ , etc.

Functions $F_{GB}(s)$ and $Q(s)$ in $Z_{GB}(s)=F_{GB}(s)-Q(s)=0$ (see further on) are also analytical functions of  $s$. 
The basis of the following argument will be the GRAM-BACKLUND extension of the Zeta function. It will be important to observe that the zeroes of the Zeta Function are determined by the equation:
\begin{equation}
F_{GB}(s)=Q(s) \label{3}
\end{equation}

where $F_{GB}(s)$ and $Q(s)$ are defined as: 
\begin{equation}
\begin{split}
F_{GB}(s)&=\frac{N^{s-1}}{s}.\underset{n=1}{\overset{N>1}{\sum }}\frac{1}{N^s}+\frac{1}{2.N.s}+\underset{\mu =1}{\overset{\infty }{\sum }}\frac{B_{2\mu }}{N^{2.\mu }.(2.\mu )!}.(s+1)\ldots (s+2.\mu -2) \\
Q(s)&=\frac{1}{s.(1-s)}
\end{split} \label{4}
\end{equation}

It can be easily verified that \eqref{3} with definitions \eqref{4} are equivalent to $Z_{GB}(s)=0$ for $s\neq0$.
The symmetry properties of $F_{GB}(s)$ and of $Q(s)$ with respect to the C.L. will be seen to play a capital role in the following discussion. It is apparent, in this connection, that $Q(s)$ is highly symmetrical, see Figures ~\ref{fig:1}, ~\ref{fig:2} and ~\ref{fig:3}:
\begin{equation}
\begin{split}
Q(1-s)&=Q(s) \\
\overline{Q}(1-s)&=Q\left(\overline{s}\right)=\overline{Q}(s) 
\end{split} \label{5}
\end{equation}

while $F_{GB}(s)$ is the sum of several terms with widely different symmetries and anti-symmetries; therefore:	

\begin{equation}
F_{GB}(s)=F_{GB}^S(s)+F_{GB}^{AS}(s) \label{6}
\end{equation}

where symmetry and anti-symmetry with respect to the C.L. are defined as follows:

\begin{defn}
Symmetry
\begin{equation}
\begin{split}
F_{GB}^S(1-s)=F_{GB}^S(s) &\text{ and } \overline{F}_{GB}^S(1-s)=F_{GB}^S\left(\overline{s}\right)=\overline{F}_{GB}^S(s) \\
Q(1-s)=Q(s) &\text{ and } \overline{Q}(1-s)=Q\left(\overline{s}\right)=\overline{Q}(s) 
\end{split} \label{7}
\end{equation}
\end{defn}

\begin{defn}
Anti-symmetry
\begin{equation}
F_{GB}^{AS}(1-s)=-F_{GB}^{AS}(s) \text{ and } \overline{F}_{GB}^{AS}(1-s)=-F_{GB}^{AS}\left(\overline{s}\right)=-\overline{F}_{GB}^{AS}(s) \label{8}
\end{equation}
\end{defn}

%
%

\section{The hypothetical outlying zeroes: the H/DLVP theorem}\label{sec:3}

HADAMARD and DE LA VALL\'EE POUSSIN showed , see Refs. \cite{Hada96} and \cite{Dela96}, that the outlying zeroes, if they exist, must occur in duplicate form, i.e. as two c.c. pairs located at mirror-images of each other with respect to the C.L., see Figure ~\ref{fig:4}. In the following any such hypothetical foursome of outlying zero-points will be called a ‘H/DLVP quartet’, and the member of the quartet with $X-\frac{1}{2}>0$ , $Y>0$ will be taken as representative of the whole quartet: indeed, if 
\begin{equation}
s=\frac{1}{2}+\rho.e^{i.\alpha } \label{9}
\end{equation}

it comes 
\begin{defn}
\begin{equation}
\begin{split}
s&=X+i.Y \\
X-\frac{1}{2}&=\varepsilon =\rho.\cos \alpha \\
Y&=\eta=\rho.\sin \alpha 
\end{split}\label{10}
\end{equation}
with the constraints:
\begin{equation}
\begin{split}
0&<\cos \alpha <\frac{1}{2.\rho} \\
\sqrt{1-\frac{1}{4.\rho ^2}}&<\sin \alpha <1 
\end{split}\label{11}
\end{equation}
\end{defn}

It is evident, see Figure ~\ref{fig:4}, that the four members of the hypothetical quartet lie on a circle of radius $\rho$ and center located at $s=\frac{1}{2}$ , called "circle $\Gamma$" in the following. At the generic point on the circle $\Gamma$ it is, with $0\leq \vartheta \leq 2.\pi$:	
\begin{equation}
s=\frac{1}{2}+\rho.e^{i.\vartheta } \label{12}
\end{equation}

The member representative of the hypothetical quartet lies at azimuth $\vartheta=\alpha$ on that circle, with:
\begin{equation}
arcos\frac{1}{2.\rho }<\alpha <\frac{\pi }{2} \label{13}
\end{equation}

%
%
\section{Development of functions $Q(s)$ and $F_{GB}(s)$ in series of powers of $\rho.e^{i.\vartheta }$ }\label{sec:4}

The functions $F_{GB}(s)$ and $Q(s)$ are analytic functions of the variable $s$. Along the circle $\Gamma$ it is:
\begin{equation}
s=\frac{1}{2}+\rho.e^{i.\vartheta } \label{14}
\end{equation}

And defining a new variable $s'$:
\begin{equation}
s'=s-\frac{1}{2}=\rho.e^{i.\vartheta } \label{15}
\end{equation}
which is an analytic function of $s$, along the circle $\Gamma$ the functions $F_{GB}(s)$ and $Q(s)$ become analytic functions of the variable $s'$: 
\begin{equation}
F_{GB}(s')=F_{GB}\left(\rho.e^{i.\vartheta }\right) \text{ and } Q(s')=Q\left(\rho.e^{i.\vartheta }\right) \label{16}
\end{equation}
which are continuous and periodical over $\Gamma$:
\begin{equation}
F_{GB}\left(\rho.e^{i.(\vartheta \pm 2.k.\pi )}\right)=F_{GB}\left(\rho.e^{i.\vartheta }\right) \text{ and } Q\left(e^{i.(\vartheta \pm 2.k.\pi )}\right)=Q\left(\rho.e^{i.\vartheta }\right) \label{17}
\end{equation}
with $k\in \mathbb{I}$ (natural integer, zero included). 

It is therefore legitimate to develop these two functions as power series of $\rho.e^{i.\vartheta }$
\begin{equation}
F_{GB}\left(\rho.e^{i.\vartheta }\right)=F_{GB}^S\left(\rho.e^{i.\vartheta }\right)+F_{GB}^{AS}\left(\rho.e^{i.\vartheta }\right) \label{18}
\end{equation}
see Eqs. \eqref{6}, \eqref{7} , where the symmetric and anti-symmetric components of $F_{GB}\left(\rho.e^{i.\vartheta }\right)$ are respectively:
\begin{equation}
\begin{split}
F_{GB}^S\left(\rho.e^{i.\vartheta }\right)&=\underset{m=-\infty }{\overset{+\infty }{\sum }}C_{2.m}.\rho ^{2.m}.e^{i.(2.m.\vartheta )} \\
F_{GB}^{AS}\left(\rho.e^{i.\vartheta }\right)&=\underset{m=-\infty }{\overset{+\infty }{\sum }}C_{2.m+1}.\rho ^{2.m+1}.e^{i.(2.m+1).\vartheta } 
\end{split}\label{19}
\end{equation}

where $m\in \mathbb{I}$ (natural integer, zero included), and the coefficients $C_{2.m}$ , $C_{2.m+1}$ are obtained by integration over the circumference of circle $\Gamma$:	
\begin{equation}
\begin{split}
C_{2.m}&=\frac{1}{2.\pi }.\int _0^{2.\pi }F_{GB}\left(\rho.e^{i.\vartheta }\right).\rho ^{-2.m}.e^{-i.(2.m.\vartheta )}.d\vartheta \\
C_{2.m+1}&=\frac{1}{2.\pi }.\int _0^{2.\pi }F_{GB}\left(\rho.e^{i.\vartheta }\right).\rho ^{-2.m+1}.e^{-i.(2.m+1).\vartheta }.d\vartheta 
\end{split}\label{20}
\end{equation}
Note that it is:	
\begin{equation}
\begin{split}
\int _0^{2.\pi }F_{GB}\left(\rho.e^{i.\vartheta }\right).\rho ^{-2.m}.e^{-i.(2.m.\vartheta )}.d\vartheta &=\int _0^{2.\pi }F_{GB}^S\left(\rho.e^{i.\vartheta }\right).\rho ^{-2.m}.e^{-i.(2.m.\vartheta )}.d\vartheta \\
\int _0^{2.\pi }F_{GB}\left(\rho.e^{i.\vartheta }\right).\rho ^{-2.m+1}.e^{-i.(2.m+1).\vartheta }.d\vartheta &=\int _0^{2.\pi }F_{GB}^{AS}\left(\rho.e^{i.\vartheta }\right).\rho ^{-2.m+1}.e^{-i.(2.m+1).\vartheta }.d\vartheta
\end{split}\label{21}
\end{equation}
because: 
\begin{equation}
\begin{split}
\int _0^{2.\pi }F_{GB}^{AS}\left(\rho.e^{i.\vartheta }\right).\rho ^{-2.m}.e^{-i.(2.m.\vartheta )}.d\vartheta =0 \\
\int _0^{2.\pi }F_{GB}^S\left(\rho.e^{i.\vartheta }\right).\rho ^{-2.m+1}.e^{-i.(2.m+1).\vartheta }.d\vartheta =0 
\end{split}\label{22}
\end{equation}
so that:
\begin{cor}
the symmetric component of $F_{GB}\left(\rho.e^{i.\vartheta }\right)$ is expressed as the sum of powers of $\rho.e^{i.\vartheta }$ with even exponents , see Eqs. \eqref{7} 
\end{cor}
and
\begin{cor}
the anti-symmetric component of $F_{GB}\left(\rho.e^{i.\vartheta }\right)$ is expressed as the sum of powers of $\rho.e^{i.\vartheta }$ with odd exponents , see Eq. \eqref{8} 
\end{cor}

Analogously, for the function $Q(s)=\frac{1}{s.(1-s)}$:
\begin{equation}
Q\left(\rho.e^{i.\vartheta }\right)=\underset{m=-\infty }{\overset{+\infty }{\sum }}D_{2.m}.\rho ^{2.m}.e^{i.(2.m.\vartheta )} \label{23}
\end{equation}
where: 
\begin{equation}
D_{2.m}=\frac{1}{2.\pi }.\int _0^{2.\pi }Q\left(\rho.e^{i.\vartheta }\right).\rho ^{-2.m}.e^{-i.(2.m.\vartheta )}.d\vartheta \label{24}
\end{equation}
since $Q\left(\rho.e^{i.\vartheta }\right)$ is symmetrical with respect to the C.L. in the above-defined sense, see Eqs. \eqref{5} to \eqref{8}\footnote{From the definitions of "symmetric anti-symmetric functions", see Eqs. \eqref{7}, \eqref{8} and from their respective properties, see Eqs. \eqref{20} to \eqref{23}, as well as from the definitions of coefficients $C_{2.m}$, $C_{2.m+1}$ ,$D_{2.m}$ , see again Eqs. \eqref{20} to \eqref{23}, it can be easily proved that all these coefficients are real quantities functions of $\rho$ and $m$.}.

%
%
\section{Expression of the null-conditions for $Z_{GB}\left(\rho.e^{i.\vartheta }\right)$ in the hypothetical H/DLVP zero-point at $\vartheta =\alpha$ }\label{sec:5}

Assume now that a H/DLVP quartet of outlying zeroes has been found at radius $\rho$ and azimuth $\vartheta = \alpha$: the null-conditions of $Z_{GB}\left(\rho.e^{i.\alpha }\right)$ are to be expressed as follows:		
\begin{equation}
F_{GB}\left(\rho.e^{i.\alpha }\right)=F_{GB}^{AS}\left(\rho.e^{i.\alpha }\right)+F_{GB}^S\left(\rho.e^{i.\alpha }\right)=Q\left(\rho.e^{i.\alpha }\right) \label{25}
\end{equation}
i.e. , after Eqs. \eqref{19}, \eqref{23}: 
\begin{equation}
\begin{split}
&\underset{m=-\infty }{\overset{+\infty }{\sum }}C_{2.m+1}.\rho ^{2.m+1}.e^{i.(2.m+1)\alpha }+\underset{m=-\infty }{\overset{+\infty }{\sum }}C_{2.m}.\rho ^{2.m}.e^{i.(2.m.\alpha )}= \\
&\underset{m=-\infty }{\overset{+\infty }{\sum }}D_{2.m}.\rho ^{2.m}.e^{i.(2.m.\alpha )} 
\end{split} \label{26}
\end{equation}
In order to fulfill this condition, it must be: 
\begin{equation}
F_{GB}^{AS}\left(\rho.e^{i.\alpha }\right)=0 \text{ and } F_{GB}^S\left(\rho.e^{i.\alpha }\right)=Q\left(\rho.e^{i.\alpha }\right) \label{27}
\end{equation}
i.e. Eq. \eqref{25} must hold separately for the anti-symmetric and symmetric components of $Z_{GB}\left(\rho.e^{i.\alpha }\right)$, because otherwise, as a consequence of Definitions 1 and 2 , see Eqs. \eqref{7} and \eqref{8}, in each of the two pairs of c.c. zero-points (each pair consisting of two points mirroring each other across the C.L.), the zero condition, if fulfilled at one of them (say the point to the left of the C.L.) would be violated at its mirror-image (say the point to the right of the C.L.). In terms of the power-series representation of the functions involved in Eq. \eqref{26}, Eq. \eqref{27} therefore yields the two conditions: 	
\begin{equation}
\begin{split}
\underset{m=-\infty }{\overset{+\infty }{\sum }}C_{2.m+1}.\rho ^{2.m+1}.e^{i.(2.m+1)\alpha }&=0 \\
\underset{m=-\infty }{\overset{+\infty }{\sum }}C_{2.m}.\rho ^{2.m}.e^{i.(2.m.\alpha )}&=\underset{m=-\infty }{\overset{+\infty }{\sum }}D_{2.m}.\rho ^{2.m}.e^{i.(2.m.\alpha )} 
\end{split} \label{28}
\end{equation}
The first of the null-conditions \eqref{28} is conveniently broken down into its real and imaginary parts: 
\begin{equation}
\begin{split}
\underset{m=-\infty }{\overset{+\infty }{\sum }}C_{2.m+1}.\rho ^{2.m+1}.\cos [(2.m+1).\alpha ]&=0 \\
\underset{m=-\infty }{\overset{+\infty }{\sum }}C_{2.m+1}.\rho ^{2.m+1}.\sin [(2.m+1).\alpha ]&=0 
\end{split} \label{29}
\end{equation}
These are two necessary (but by no means sufficient) conditions to be fulfilled at the representative point of the assumed H/DLVP quartet. In the next section, the implications of Eqs. \eqref{28}, \eqref{29} will be worked out.

%
%
\section{Implications of the null-conditions for the anti-symmetric component of $Z_{GB}\left(\rho.e^{i.\alpha }\right)$}\label{sec:6}

In order to discuss this point, it is necessary to prove an important theorem. In exposing, and tentatively proving, such a theorem the sum of terms with odd positive exponents will be treated separately. Analogous treatment holds, \textit{mutatis mutandis}, for the sum of anti-symmetric powers with odd negative exponents.

\begin{theorem}
The null-condition for the anti-symmetric component of $Z_{GB}$ cannot be fulfilled if $\alpha$, $\rho.cos\alpha$, $\rho.sin\alpha$ are constrained to be real variables.
\end{theorem}

\begin{proof}
Assuming that an H/DLVP quartet exists, consider the relevant $\Gamma$-circle containing the four zero-points. Along the circle the radius $\rho$ is a constant which is assumed to be known. One of the conditions to be fulfilled at the representative point of the quartet will be sorted out to determine the azimuth of this point, i.e. $\alpha$.

Let us write \eqref{29} in complex algebra notation: recalling that the null-condition must hold both in the representative point of the H/DLVP quartet and in its mirror-image with respect to the C.L., it is:
\begin{equation}
\begin{split}
\rho.(\pm \cos \alpha +i.\sin \alpha )&=\varepsilon +i.\eta \\
\cos \alpha &=\frac{1}{\sqrt{1+\tan ^2\alpha }}=\frac{1}{\sqrt{1+i.\tan \alpha }}.\frac{1}{\sqrt{1-i.\tan \alpha }}\\
\sin \alpha &=\frac{\tan \alpha }{\sqrt{1+\tan ^2\alpha }}\\
\varepsilon +i.\eta=\rho .(\pm \cos \alpha +i.\sin \alpha )&=\rho .\frac{\pm 1+i.\tan \alpha }{\sqrt{1+\tan ^2\alpha }}=\left\{
\begin{array}{c}
 \rho .\sqrt{\frac{1+i.\tan \alpha }{1-i.\tan \alpha }} \\
 \text{or else:}\\
-\rho .\sqrt{\frac{1-i.\tan \alpha }{1+i.\tan \alpha }}
\end{array}
\right.
\end{split} \label{30}
\end{equation}

(the last written formula,  $\varepsilon +i.\eta =-\rho .\sqrt{\frac{1-i.\tan \alpha }{1+i.\tan \alpha }}$, is particularly important because it derives from the H/DLVP condition of specularity of the outlying zeroes). But also, with $t=2.m$ or $t=2.m+1$: 

\begin{equation*}
\rho ^t.e^{i.t.\alpha }=\rho ^t.[\cos (t.\alpha )+i.\sin (t.\alpha )]=\rho ^t.\frac{1+i.\tan (t.\alpha )}{\sqrt{1+\tan ^2(t.\alpha )}}=\rho ^t.\left(\frac{\pm 1+i.\tan \alpha }{\sqrt{1+\tan ^2\alpha }}\right)^t 
\end{equation*}

so that it is immediate to put Eqs. \eqref{29}  under one or the other of the following complex forms : 
\begin{equation} \label{31}
-\rho .\sqrt{\frac{1-i.\tan \alpha }{1+i.\tan \alpha }}.\underset{m=0}{\overset{\infty }{\sum }}C_{2.m+1}.\rho ^{2.m}.\left(\frac{1-i.\tan \alpha }{1+i.\tan \alpha }\right)^m=0  
\end{equation}

\begin{equation} \label{32}
\rho .\sqrt{\frac{1+i.\tan \alpha }{1-i.\tan \alpha }}.\underset{m=0}{\overset{\infty }{\sum }}C_{2.m+1}.\rho ^{2.m}.\left(\frac{1+i.\tan \alpha }{1-i.\tan \alpha }\right)^m=0
\end{equation}
both equivalent to:
\begin{equation}
\underset{m=0}{\overset{\infty }{\sum }}C_{2.m+1}.\rho ^{2.m+1}.\frac{1+i.\tan [(2.m+1).\alpha ]}{\sqrt{1+\tan ^2[(2.m+1).\alpha ]}}=0  \label{33}
\end{equation}
The condition that  $\tan \alpha $  and  $\tan (t.\alpha)$,  with  $t=\text{any arbitrary positive integer}$, be univocally determined is:  $\tan (t.\alpha )=\tan [(t+1).\alpha]$  , i.e. $\frac{\tan (t.\alpha )+\tan \alpha }{1-\tan (t.\alpha ).\tan \alpha }=\tan (t.\alpha )$ from which  $\tan (t.\alpha )=\tan \alpha =\mp i$ , so that in order to comply with  Eqs. \eqref{31}, \eqref{32}  keeping all their terms real, there are two possible (complex) finite options for  $\tan \alpha$, besides the (infinite) real option leading to a canonical zero : 	

\begin{equation}
\begin{split}
(\tan \alpha &=\pm \infty )\\
\tan \alpha &=\mp i 
\end{split} \label{34}
\end{equation}

If it were acceptable to give $\varepsilon$ a complex value, the option  $\tan \alpha = -i$ would give:
\begin{equation}
\varepsilon =\frac{\eta }{\tan \alpha }=i.\eta   \label{35}
\end{equation}

which yields the entirely acceptable solution  $\varepsilon +i.\eta =2.i.\eta =i.Y$, leading again to a canonical zero, while the option $\tan \alpha = +i$  would give:
\begin{equation}
\varepsilon =-i.\eta    \label{36}
\end{equation}

which is to be rejected because it gives   $\varepsilon +i.\eta =0$,  i.e. the point  $s=\frac{1}{2}$  where   $Z_{GB}\neq 0$.

Moreover, if  $\tan \alpha = -i$, it is: 
\begin{equation}
\tan (\alpha )=\tan (2.\alpha )=\tan (3.\alpha )=\ldots =\tan (2.m.\alpha )=\tan \{(2.m+1).\alpha \}=-i    \label{37}
\end{equation}
so that Eq.\eqref{31} is satisfied while the option expressed by  Eq.\eqref{32} is to be discarded.	

But is the solution $\tan \alpha =\mp i$  unique?
Assume that  besides  $\alpha =\alpha _1$ there is on circle   $\Gamma$  another solution  $\alpha =\alpha _2\neq \alpha _1$. Then it should be :
\begin{equation*}
\underset{m=0}{\overset{\infty }{\sum }}C_{2.m+1}.\rho ^{2.m}.\left[\left(\sqrt{\frac{1-i.\tan \alpha _1}{1+i.\tan \alpha _1}}\right){}^{2.m}-\left(\sqrt{\frac{1-i.\tan \alpha _2}{1+i.\tan \alpha _2}}\right){}^{2.m}\right]=0
\end{equation*}
from which either (condition A):
\begin{equation*}
\left(\sqrt{\frac{1-i.\tan \alpha _1}{1+i.\tan \alpha _1}}\right){}^m-\left(\sqrt{\frac{1-i.\tan \alpha _2}{1+i.\tan \alpha _2}}\right){}^m=0
\end{equation*}
which yields  $\tan \alpha _2=\tan \alpha _1$, or (condition B):
\begin{equation*}
\left(\sqrt{\frac{1-i.\tan \alpha _1}{1+i.\tan \alpha _1}}\right){}^m+\left(\sqrt{\frac{1-i.\tan \alpha _2}{1+i.\tan \alpha _2}}\right){}^m=0
\end{equation*}
which yields  $\tan \alpha _2=-\frac{1}{\tan \alpha _1}$   if  $m=\text{odd integer}$  but $\tan \alpha _2=\tan \alpha _1$  if  $m=\text{even integer}$. This dependence of  $\tan \alpha _2$  from the parity of exponent  $m$  is an inconsistency; besides, $\tan \alpha _2=-\frac{1}{\tan \alpha _1}$ would be inconsistent also with condition A. Thus a second outlying zero cannot exist unless $\tan \alpha _1=-\frac{1}{\tan \alpha _1}$, from which  $\tan ^2\alpha _1=-1$ ,  $\tan \alpha _1=\tan \alpha _2=\pm i$ , which means that the assumed second "outlying zero" cannot be distinct from the first one. The latter, on the other hand, is not distinct from a "canonical" zero .  The necessary conditions A and B  are therefore satisfied by $\tan \alpha _1=\tan \alpha _2=\mp i$, but not by distinct real-valued solutions for  $\tan \alpha _1$  and $\tan \alpha _2$ respecting the other constraints of Eq.\eqref{11} on $\tan \alpha$.

The necessary conclusion of this analysis seems to be that outlying zeroes cannot exist. 
\end{proof}

Of course, to actually determine the value of $\rho=Y$ at the canonical zero one would have to take into account also the two remaining equations pertaining to the real and the imaginary components of $Z_{GB}^S$ and of $Q$, see Section \ref{sec:7}.
It can be shown that making $\alpha =i.\beta $ with $\beta $ real and developing formally the above results one would get an internally consistent canonical solution of $Z_{GB}^S=0$.

It is interesting per se to observe that the equation $\varepsilon =\pm i.\eta $ means that the assumed zero-point must lie on one of the two "isotropic lines" issuing from $s=\frac{1}{2}$; besides, this circumstance is also susceptible of an interesting interpretation.

Let us operate a change of Cartesian reference axes with origin at $s=\frac{1}{2}$. This change will consist in a counterclockwise rotation of the reference axes through an angle $\gamma$. Then the new coordinates $\varepsilon ',\eta '$ of our assumed H/DLVP zero will become: 
\begin{equation}
\begin{split}
\varepsilon '&=\varepsilon.\cos \gamma +\eta.\sin \gamma \\
\eta '&=-\varepsilon.\sin \gamma +\eta.\cos \gamma 
\end{split} \label{38}
\end{equation}
and the zero-point located at $\varepsilon =\rho.\cos \alpha =i.\eta$ , $\eta =\rho.\sin \alpha $ with respect to the old axes, will take the new coordinates:
\begin{equation}
\begin{split}
\varepsilon '&=\rho.(i.\sin \alpha.\cos \gamma +\sin \alpha.sin\gamma) \\
\eta '&=-\rho.(i.\sin \alpha.\sin \gamma -\sin \alpha.cos\gamma) \\
\varepsilon '&=i.\eta ' 
\end{split} \label{39}
\end{equation}
The null-condition for the anti-symmetric component of $Z_{GB}^{AS}$, i.e. for the sum of the terms with odd-exponents of $\rho.e^{i.\vartheta '}$, is obtained imposing:
\begin{equation}
\begin{split}
\rho.\cos \alpha '.\Omega '_c-\rho.\sin \alpha '.\Omega '_s&=0 \\
\rho.\sin \alpha '.\Omega '_c+\rho.\cos \alpha '.\Omega '_s&=0
\end{split} \label{40}
\end{equation}
where $\alpha '=\alpha -\gamma$ and $\Omega '_c=\cos \gamma.\Omega _c-\sin \gamma.\Omega _s$, $\Omega '_s=\sin \gamma.\Omega _c+\cos \gamma.\Omega _s$, $\Omega _c$ and $\Omega _s$ given by \footnote{The accented variables, which are defined in function of the new independent variable $\rho.e^{i.\vartheta '}$ as the non-accented variables were defined in function of the old independent variable $\rho.e^{i\vartheta }$, are obtained by making in the old definitions $\alpha =\alpha '+\gamma $ , $\vartheta =\vartheta '+\gamma$ and applying the usual trigonometric transformations. Details are here omitted for brevity.}:
\begin{equation}
\begin{split}
\Omega _c=\underset{m=0}{\overset{+\infty }{\sum }}C_{2.m+1}.\rho ^{2.m}.\cos (2.m.\alpha )&=0 \\
\Omega _s=\underset{m=0}{\overset{+\infty }{\sum }}C_{2.m+1}.\rho ^{2.m}.\sin (2.m.\alpha )&=0
\end{split} \label{41}
\end{equation}
And again, putting the above in complex form as in Eqs. \eqref{31}, \eqref{32}, \eqref{33} it should be:
\begin{equation}
\left(\varepsilon '-i.\eta '\right).\left(\varepsilon '+i.\eta '\right)=0 \label{42}
\end{equation}
and of the two possible options the one to be retained should be: 
\begin{equation}
\begin{split}
\tan \alpha '&=-i \\
\varepsilon '&=i.\eta '
\end{split} \label{43}
\end{equation}
exactly as in the old system of reference; the same implications therefore again apply.
The third of Eq. \eqref{39} and Eq. \eqref{42} mean, indeed, that the assumed H/DLVP zero-point lies on one or the other of the two isotropic lines issuing from the origin of the new system of reference, as eq. \eqref{34} mean that it lies on one of the two isotropic lines issuing from the origin of the old system of reference. Thus it enjoys the characteristic property of the isotropic lines, which are the fixed lines of the rotation group. Apparently one would get also $\rho =\sqrt{\left(\varepsilon '-i.\eta '\right).\left(\varepsilon '+i.\eta '\right)}=0$ , which is another peculiar feature of the isotropic lines (property of null length).
All of the foregoing considerations seem to be consistent indications that the only c.c. pairs of zero-points of $Z_{GB}$ are on the C.L. (canonical zeroes). Outlying zero-points such as those foreseen by H/DLVP are, apparently, not allowed as a consequence of the symmetry properties of the term $Q(s)=\frac{1}{s.(1-s)}$ in the R. H. S. of the second of Eq. \eqref{4}.

%
%
\section{Remarks on the methodology of approach and on the results obtained thus far}\label{sec:7}

The necessary conditions applied so far in order to obtain a non-trivial outlying zero are the following ones: 
-in the real domain:
\begin{equation}
\begin{split}
\rho.\cos \alpha.\underset{m=-\infty }{\overset{+\infty }{\sum }}C_{2.m+1}.\rho ^{2.m}.\cos (2.m.\alpha )&-\rho.\sin \alpha.\underset{m=-\infty }{\overset{+\infty }{\sum }}C_{2.m+1}.\rho ^{2.m}.\sin (2.m.\alpha )=0 \\
\rho.\sin \alpha.\underset{m=-\infty }{\overset{+\infty }{\sum }}C_{2.m+1}.\rho ^{2.m}.\cos (2.m.\alpha )&+\rho.\cos \alpha.\underset{m=-\infty }{\overset{+\infty }{\sum }}C_{2.m+1}.\rho ^{2.m}.\sin (2.m.\alpha )=0 \\
|\varepsilon =\rho.\cos \alpha|&<\frac{1}{2} \\
\eta =\rho.\sin \alpha &>0
\end{split} \label{44}
\end{equation}
-upon observing that these conditions may lead to inconsistencies in the real domain, it was attempted to allow $\alpha$ to be a complex quantity:
\begin{equation}
\alpha =i.\beta \label{45}
\end{equation}
where $\beta$ is real. More in general, it would be possible to assume that $\varepsilon$ and/or $\eta$ be allowed to take on complex values, but this would lead to the same end results at the price of more involved elaborations. On the other hand, it is convenient to assume that $\eta$ is real and positive, since the hypothetical H/DLVP zero-point should be above the $X$ axis. Therefore the choice was made to assume $\varepsilon =K.\eta$, which does not introduce any new constraints, and let $K$ be determined by the condition that the system \eqref{44} of two equations be not constrained to have null solutions. $K$ was allowed to be a complex quantity so that the system \eqref{44} might have a supplementary degree of freedom. 
In this way it was found that it must be: 
\begin{equation}
\begin{split}
K&=\pm i \\
\varepsilon &=i.\eta \\
\varepsilon +i.\eta &=2.i.\eta =i.Y
\end{split} \label{46}
\end{equation}
i.e. the hypothetical H/DLVP zero-point must coincide with a canonical zero.
The move to the complex domain for $\varepsilon$, thus, appears as a temporary short-cut because in the end, reverting to the real domain, the solution found is equivalent to:
\begin{equation}
\begin{split}
\varepsilon &= 0 \\
\eta &= Y
\end{split} \label{47}
\end{equation}
The Authors were intrigued to observe that this solution implies that $\varepsilon +i.\eta $ belongs to one of the two isotropic lines issuing from $s=\frac{1}{2}$. This fact, indeed, was found to be consistent with a rotation of the coordinate system around the point $s=\frac{1}{2}$, which leaves the new complex quantity $\varepsilon '+i.\eta '$ unchanged with respect with the old value $\varepsilon +i.\eta $; therefore the value of $Q'(\alpha ')=\frac{1}{\left(\frac{1}{2}+\varepsilon '+i.\eta '\right).\left(\frac{1}{2}-\varepsilon '-i.\eta '\right)}$ remains the same as $Q(\alpha )=\frac{1}{\left(\frac{1}{2}+\varepsilon +i.\eta \right).\left(\frac{1}{2}-\varepsilon -i.\eta \right)}$ and likewise the value of $Z_{GB}^S\left(\frac{1}{2}+\varepsilon +i.\eta \right)=Q(\alpha )$ remains the same as $Z_{GB}^S\left(\frac{1}{2}+\varepsilon '+i.\eta '\right)=Q'(\alpha ')$.
This unexpected relevance of a concept of Projective Geometry -the isotropic lines- to the analysis of the RIEMANN Hypothesis is deemed by the Authors to be worthy of note and possibly deserving further investigations.

%
%
\section{The remaining conditions for the determination of actual solutions of $Z_{GB}=0$}\label{sec:8}

The tentative result that all the non-trivial zeroes of $Z_{GB}$ lie along the C.L. was reached in the preceding sections by considering the conditions to be fulfilled by the anti-symmetric components of $Z_{GB}$. To complete the treatment, it remains now to see how the full conditions are expressed in our methodology of approach.

\begin{description}
\item[A]The angle must be $\alpha =\frac{\pi }{2}$. Therefore $\cos \alpha =0$ and $\sin \alpha =1$.
\item[B]As a consequence of point A, the first of Eq. \eqref{44} is identically fulfilled. The second of Eq. \eqref{44} becomes:
\begin{equation}
\underset{m=-\infty }{\overset{+\infty }{\sum }}C_{2.m+1}.\rho ^{2.m+1}.(-1)^m=\underset{m=-\infty }{\overset{+\infty }{\sum }}C_{2.m+1}.\rho ^{2.m+1}.\sin \left[(2.m+1).\frac{\pi }{2}\right]=0 \label{48}
\end{equation}
Since the coefficients $C_{2.m+1}$ are functions of $\rho$, Eq. \eqref{48} will be fulfilled at isolated points along the C.L.
\item[C]The equations involving the symmetric components of $Z_{GB}$ become:
\begin{subequations}
\begin{align}
\underset{m=-\infty }{\overset{+\infty }{\sum }}C_{2.m}.\rho ^{2.m}.\cos \left(2.m.\frac{\pi }{2}\right)=\frac{1}{\left(\frac{1}{2}+i.\rho \right).\left(\frac{1}{2}-i.\rho \right)}&=-\underset{m=0}{\overset{\infty }{\sum }}\frac{\rho ^{-2.(m+1)}.\cos \left[2.(m+1).\frac{\pi }{2}\right]}{4^m} \label{eq49a} \\
\underset{m=-\infty }{\overset{0}{\sum }}C_{2.m}.\rho ^{2.m}.\sin \left(2.m.\frac{\pi }{2}\right)&\equiv 0\label{eq49b}
\end{align}
\end{subequations}
Eq. \eqref{eq49b} is an identity. Since in Eq. \eqref{eq49a} the coefficients $C_{2.m}$ are functions of $\rho$, Eq. \eqref{eq49a} will be fulfilled at isolated points along the C.L.
\end{description}
The points of the C.L. in which both eqs. \eqref{48} and \eqref{eq49a} are fulfilled will be the canonical zeroes of $Z_{GB}$.

%
%
\section{Conclusions}\label{sec:6}
The analysis carried out of the algebraic structure of $Z_{GB}$, of the necessary conditions implied by $Z_{GB}(s)=F_{GB}(s)-Q(s)=0$ and of the symmetry properties possessed by $F_{GB}(s)$ and $Q(s)$ seem to point out that the application of such necessary conditions to functions endowed with those symmetry properties is incompatible with the existence of outlying zeroes. If this conclusion were confirmed, it would provide, in the Authors’ opinion, an interesting line of approach to a full, rigorous discussion of the RIEMANN Hypothesis.

In the course of the analysis, it became apparent that the incompatibilities between the null-conditions for $Z_{GB}(s)$ and its symmetry features could be formally overcome moving the domain of definition of the variables involved to a wider domain. However, by translating the result thus obtained into the conditions required by the real and imaginary components of the actual variables involved, the same limitation (i. e that the allowed non-trivial zeroes must lie on the C.L.) was confirmed.

The enlargement of the variable domain thus formally introduced was not, however, a useless detour, insofar as it allowed to highlight an unexpected connection between the properties of the Zeta function and the properties of the isotropic lines, i.e. with entities and concepts proper to Projective Geometry. It is deemed by the Authors of the present essay that this connection, only very sketchily outlined in their analysis, is worthy of deeper investigation.

It is also remarkable that symmetry considerations were sufficient to reveal important aspects relevant to the discussion of the RIEMANN Hypothesis. The power of the concepts inherent to the theory of the symmetry groups was thus once more underscored.

\pagebreak

\begin{figure}[!htb]
\centering
\includegraphics[scale=0.6]{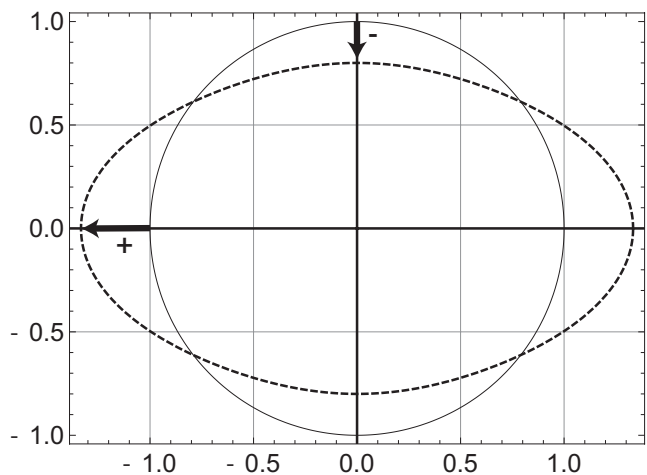}
\caption{Qualitative behavior of $Re\left[\frac{1/4}{s.(s-1)}\right]=Re\left[\underset{m=1}{\overset{\infty }{\sum }}\frac{1}{2^{2.m}}.\left(\frac{1}{\rho.e^{i.\vartheta }}\right)^{2.m}\right]$ over the circle $\Gamma$ (with radius $\rho$ reduced to a unitary value)}
\label{fig:1}
\end{figure}

\begin{figure}[!htb]
\centering
\includegraphics[scale=0.6]{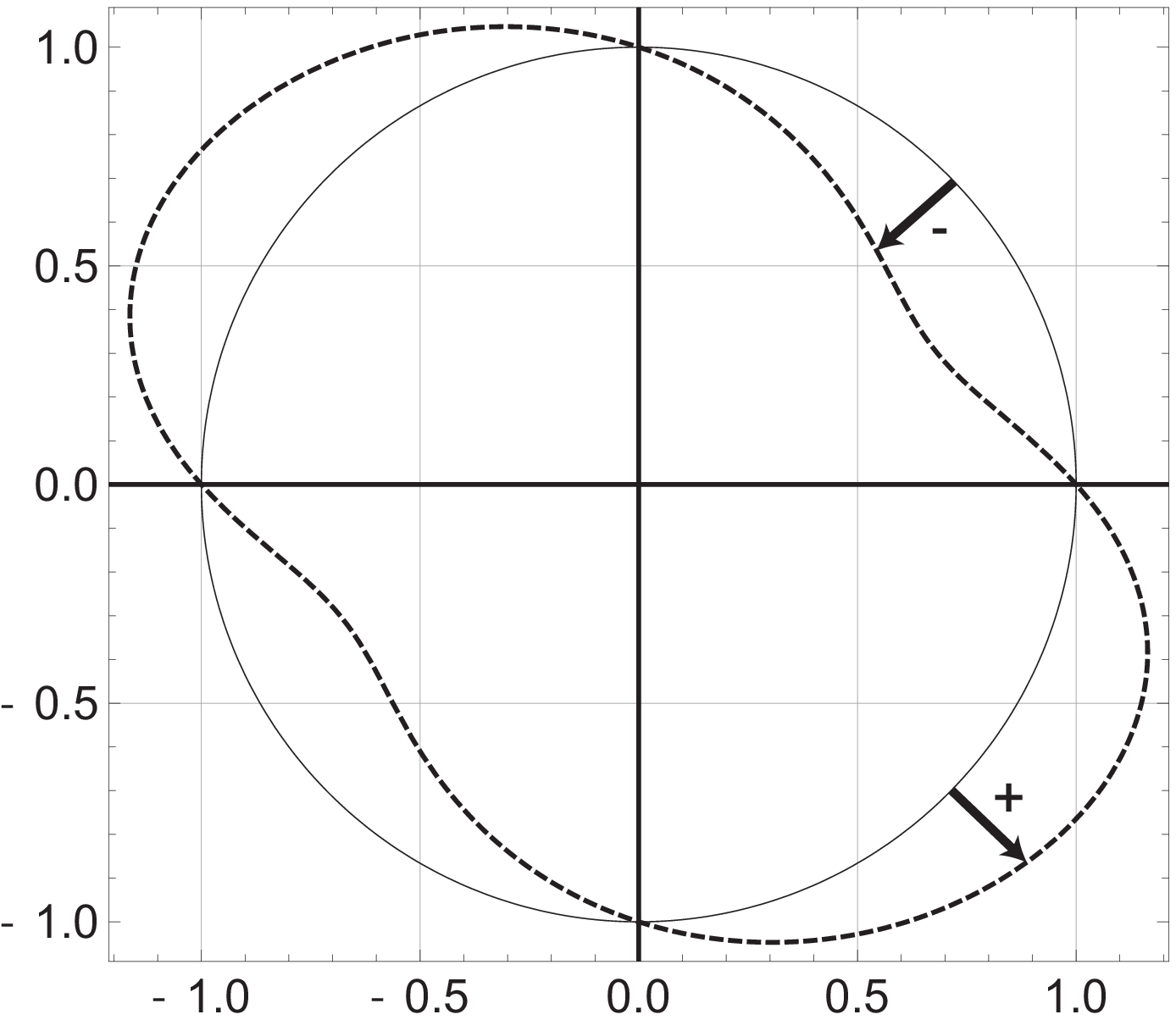}
\caption{Qualitative behavior of $Im\left[\frac{1/4}{s.(s-1)}\right]=Im\left[\underset{m=1}{\overset{\infty }{\sum }}\frac{1}{2^{2.m}}.\left(\frac{1}{\rho.e^{i.\vartheta }}\right)^{2.m}\right]$
over the circle $\Gamma$ (with radius $\rho$ reduced to a unitary value)}
\label{fig:2}
\end{figure}

\begin{figure}[!htb]
\centering
\includegraphics[scale=0.6]{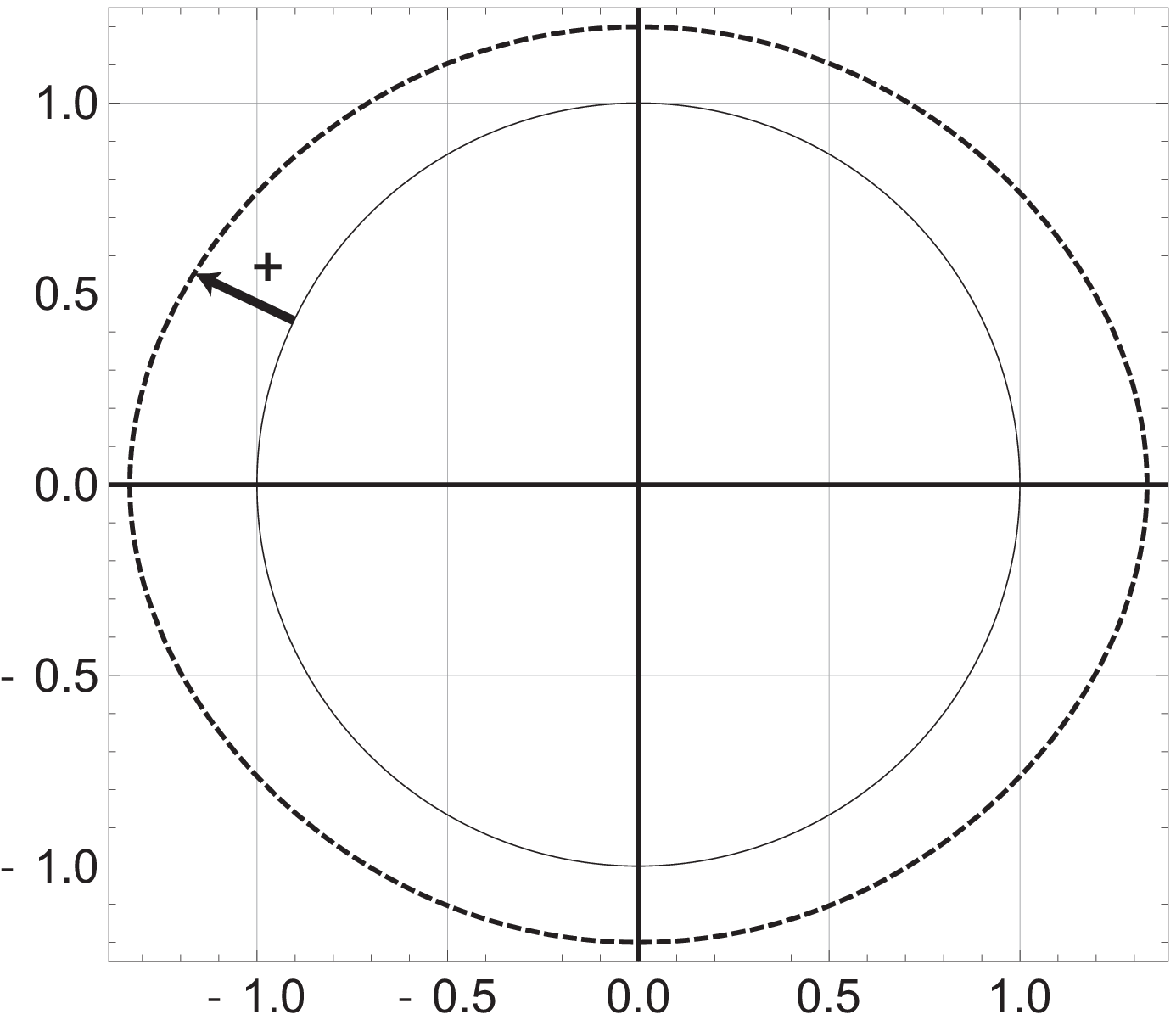}
\caption{Qualitative behavior of $\left|\underset{m=1}{\overset{\infty }{\sum }}\frac{1}{2^{2.m}}.\left(\frac{1}{\rho.e^{i.\vartheta }}\right)^{2.m}\right|$ over the circle $\Gamma$ (with radius $\rho$ reduced to a unitary value)}
\label{fig:3}
\end{figure}

\begin{figure}[!htb]
\centering
\includegraphics[scale=0.6]{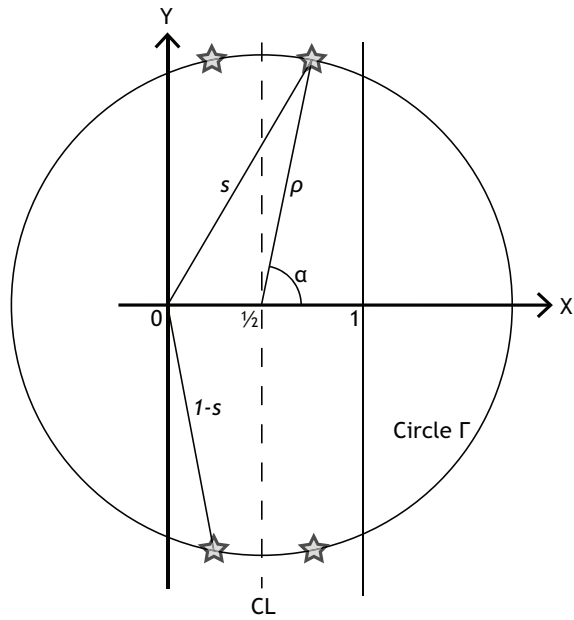}
\caption{The specularly symmetrical configuration of a H-DLVP quartet of hypothetical outlying zeroes of the Zeta function}
\label{fig:4}
\end{figure}

%
%

\end{document}